\documentclass[review]{elsarticle}

\usepackage{amssymb,amsmath,amsthm,amstext,amsfonts}
\usepackage{mathrsfs}
\usepackage{amsthm}
\usepackage{color}
\usepackage{functan}
\usepackage{graphicx}
\usepackage[latin1]{inputenc}
\usepackage{mathtools} 
\usepackage[pdftex,backref]{hyperref}
\usepackage{framed,fancybox}
\usepackage{xcolor}
\usepackage{epsfig}
\usepackage{mathrsfs}
\usepackage[latin1]{inputenc}
\usepackage{accents}

\makeatletter \@addtoreset{equation}{section} \makeatother

\renewcommand\thetable{\thesection.\@arabic\c@table}

\newtheorem{theorem}{Theorem}[section]

\newtheorem{proposition}[theorem]{Proposition}
\newtheorem{remark}[theorem]{Remark}

\newcommand{\R}{\mathbb{R}}
\newcommand{\N}{\mathbb{N}}
\newcommand{\K}{\mathbb{K}}
\newcommand{\eps}{\varepsilon}
\newcommand{\lraup}{\relbar\joinrel\rightharpoonup}
\newcommand{\bs}{\boldsymbol}
\newcommand{\cchi}{{\mbox{\raisebox{3pt}{$\chi$}}}}

\begin{document}
	\title{An evolutionary vector-valued variational inequality and Lagrange multiplier}
	\author[um]{Davide Azevedo}
	\ead{davidemsa@math.uminho.pt}
	\author[um]{Lisa Santos\corref{cor1}}
	\ead{lisa@math.uminho.pt}
	\cortext[cor1]{Corresponding author}
	\address[um]{Centro de Matem\'atica, Universidade do Minho, Campus de Gualtar, 4710-057 Braga, Portugal}


\begin{abstract} We prove existence and uniqueness of solution of an evolutionary vector-valued variational inequality defined in the convex set of vector valued functions $\bs v$ subject to the constraint $|\bs v|\le1$. We show that we can write the variational inequality as a system of equations on the unknowns $(\lambda,\bs u)$, where $\lambda$ is a (unique) Lagrange multiplier belonging to $L^p$ and $\bs u$ solves the variational inequality.
	
Given data $(\bs f_n,\bs u_{n0})$ converging to $(\bs f,\bs u_0)$ in $\bs L^\infty(Q_T)\times H^1_0(\Omega)$, we prove the convergence of the solutions $(\lambda_n,\bs u_n)$ of the Lagrange multiplier problem to the solution of the limit problem, when we  let $n\rightarrow \infty$.
\end{abstract}

\maketitle
\section{Introduction}

Variational inequalities can model different physical and biological phenomena. The obstacle problem is the best known one (see \cite{LewyStampacchia1969,Kinderlehrer1971,Rodrigues1987}), but there are also many works about variational inequalities with gradient constraint  or with other constraints on the derivatives, such as on the curl, the symmetric gradient, the bilaplacian (see \cite{Ting1966,LanchonDuvaut1967,Prigozhin1994,MirandaRodriguesSantos2012,MirandaRodriguesSantos2020,AzevedoRodriguesSantos2024}), as well as variational inequalities subject to different types of constraints.

Less known are vector variational inequalities, such as the $N$-membranes VI, in \cite{AzevedoRodriguesSantos2005, RodriguesSantosUrbano2009, SavinYu2019}, a multiphase problem, where the vector-valued solution belongs to an $N$-simplex, in \cite{RodriguesSantos2009}, a vector-valued Allen-Cahn VI, in \cite{BlankGarckeSarbuStyles2013}, and others. For completeness, we define below the first two (evolutionary) VIs referred here, in their most simplified version.

From now on, we will write VI instead of variational inequality. A vector function taking values in $\R^N$ will be denoted by a bold symbol and, if $E$ is a normed vector space, $E^N$ will be denoted by $\bs E$. If the norm in $E$ is $\|\,\cdot\,\|$, we consider, for $\bs v\in\bs E$, the norm $\|\bs v\|_{\bs E}=\|(v_1,\ldots,v_N)\|_{\bs E}=\sqrt{\|v_1\|^2+\cdots+\|v_N\|^2}$. We use $|A|$ to denote the Lebesgue measure of a measurable subset of $\R^N$.

Let $\Omega$ be an open bounded subset of $\R^N$ and, for $T>0$, let $Q_T=\Omega\times(0,T)$. The evolutionary $N$-membranes VI (\cite{RodriguesSantosUrbano2009}) consists on finding $\bs u$ such that
\begin{equation}\label{NmembVI}
	\begin{cases}
\bs u(t)\in\K_1\text{ for a.e. }t\in(0,T),\ \bs u(0)=\bs u_0,\\
\displaystyle\int_\Omega\partial_t\bs u(t)\cdot\big(\bs v-\bs u(t)\big)+\int_\Omega\nabla\bs u(t):\nabla\big(\bs v-\bs u(t)\big)\ge\int_\Omega\bs f(t)\cdot\big(\bs v-\bs u(t)\big),\qquad\forall\bs v\in\K,
	\end{cases}
\end{equation}
where, for $\bs u=(u_1,\ldots,u_N)$ and $\bs v=(v_1,\ldots,v_N)$, $\nabla\bs u:\nabla\bs v=\displaystyle\sum_{i=1}^N\nabla u_i\cdot\nabla v_i$ and
\begin{equation*}
	\K_1=\big\{\bs v=(v_1,\ldots,v_N)\in \bs H^1_0(\Omega):v_1\ge v_2\ge\cdots\ge v_N\big\}.
\end{equation*}

The evolutionary multiphase VI (\cite{RodriguesSantos2009}) consists on finding $\bs u$ satisfying  \eqref{NmembVI} with $\K_1$ replaced by
\begin{equation*}
	\K_2=\big\{\bs v=(v_1,\ldots,v_N)\in \bs H^1_0(\Omega):v_1\ge0,\ldots,v_N\ge0, v_1+\cdots+v_N\le 1\big\}.
\end{equation*}

We observe that, setting $K_1=\{(x_1,\ldots,x_N)\in\R^N:x_1\ge x_2\ge\cdots\ge x_N\}$, we have
$$\K_1=\big\{\bs v\in \bs H^1_0(\Omega):\bs v(x)\in K_1\text{ for a.e. }x\in\Omega\},$$
and, setting $K_2=\{(x_1,\ldots,x_N)\in\R^N:x_1\ge0,\ldots, x_N\ge 0, x_1+\cdots x_N\le 1\}$, then
$$\K_2=\big\{\bs v\in \bs H^1_0(\Omega):\bs v(x)\in K_2\text{ for a.e. }x\in\Omega\}.$$
The sets $K_1$ and $K_2$ are closed convex subsets of $\R^N$, which are intersection of half-spaces. As a consequence, the contact sets of the solution of the variational inequality \eqref{NmembVI} for each problem (i.e., with the convex set $\K_1$ or $\K_2$) are of ``obstacle type'' and the VI can be written as a system of equations, involving characteristic functions.

The system for the evolutionary $N$-membranes problem (\cite{RodriguesSantosUrbano2009}) is
	\begin{equation*}
		\partial_t u_i-\Delta u_i=f_i+\sum_{\mbox{\scriptsize{$1\le j<k\le N, j\le i\le
					k$}}}b^{j,k}_i[\bs f]\,\cchi_{j,k}\qquad\mbox{ a.e. in }\Omega_T,
	\end{equation*}
	where $\cchi_{j,k}$ denotes the characteristic function of each $I_{j,k}=\{(x,t)\in Q_t: u_j(x,t)=\cdots=u_k(x,t)\}$,
	\begin{eqnarray*}
		b_i^{j,k}[\bs f] & = & \left\{
		\begin{array}{ll}
			\langle \bs f\rangle_{j,k}-\langle \bs f\rangle_{j,k-1} & \mbox{ if
			}\ i=j\vspace{1mm}\\
			\mbox{} \langle \bs f\rangle_{j,k}-\langle \bs f\rangle_{j+1,k} & \mbox{ if }\ i=k\vspace{1mm}\\
			\frac{2}{(k-j)(k-j+1)}\left(\langle \bs f\rangle_{j+1,k-1}
			-\frac12( f_j+ f_k)\right) & \mbox{ if }\ j<i<k
		\end{array}\right.
	\end{eqnarray*}
	and
	$$\langle\bs f\rangle_{j,k}=\frac{f_j+\cdots+f_k}{k-j+1},\qquad 1\le j\le k\le N.$$
	We observe that, in this system, all possible contacts among the component functions of the solution $(u_1,\ldots,u_N)$ are considered. A detailed proof of the equivalence of the variational inequality and this system, in the stationary case, can be found in \cite{AzevedoRodriguesSantos2005}.
	
	The system in the multiphases problem is analogous, involving characteristic functions of all possible contacts of the solution with the different faces of the $N$-simplex. For details, see \cite{RodriguesSantos2009}.

More generally, given $\{K(x)\}_{x\in\Omega}$, a family of closed convex subsets of $\R^N$, we can define the set
$$\bs K=\big\{\bs v\in\bs H^1_0(\Omega): \bs v(x)\in K(x)\text{ for a.e. } x\in\Omega\big\}$$
and consider the VI \eqref{NmembVI} defined in this convex set. Obvious questions are the existence of solution of the VI and whether we can rewrite the VI as a system of equations. To deal with this type of problems, we need to describe the set $\bs K$ in an analytic way and this can be achieved by using the Minkowsky functional. However, the behaviours of $K(x)$ and $K(x')$ for $x$ close to $x'$ need to have some ``smoothness'' to ensure $\bs K$ contains a sufficient number of functions.

In this paper, we are considering, as an early approach to this type of problems, the convex set
\begin{equation}\label{k}
	\K=\left\{\bs v=(v_1,\ldots,v_N)\in H^1_0(\Omega):v_1^2+\cdots v_N^2 \le1\right\}.
	\end{equation}
We observe this set is defined  by a quadratic type inequality, different from the affine inequalities that define $\K_1$ and $\K_2$. We will see that the system which describes the VI in this last case is very different. In fact, the VI with the convex set \eqref{k} can be described by the system
$$\partial_t\bs u-\Delta\bs u+\lambda\bs u=\bs f\quad\text{ in }Q_T, \quad\bs u(0)=\bs u_0\quad\text{ in }\Omega,$$
where $\lambda\in L^p(Q_T)$,  for any $1\le p<\infty$, $\lambda\ge 1$ and $(\lambda-1)(|\bs u|-1)\equiv0$, where $|\bs u|^2=u_1^2+\cdots+u_N^2$.

This system is related to the equation that models the elastic-plastic torsion problem (see \cite{Brezis1972}), where the constraint $1$ is imposed on the gradient of a scalar function $u$. To be in accordance with our setting, we consider the evolutionary analogous to the elastic-plastic  problem, that can be written as follows
$$\partial_tu-\nabla\cdot(\lambda\nabla u)=f\quad\text{ in }Q_T,\quad u(0)=u_0\quad\text{ in }\Omega,$$
and $\lambda\in L^\infty(Q_T)$, $\lambda\ge 1$ and $(\lambda-1)(|\nabla u|-1)\equiv0$ (see \cite{Santos2002}).

\vspace{3mm}

In this paper, we will consider, given $\delta\in[0,\delta_0]$ with $\delta_0$ positive arbitrary, the evolutionary variational inequality: to find $\bs u_\delta$ such that
\begin{equation}\label{vi}
	\begin{cases}
		\bs u_\delta(t)\in\K\text{ for a.e. }t\in(0,T),\ \bs u_\delta(0)=\bs u_0\in\K,\\
		\displaystyle\int_\Omega\partial_t\bs u_\delta(t)\cdot(\bs v-\bs u_\delta(t)\big)+\int_\Omega\nabla\bs u_\delta(t):\nabla\big(\bs v-\bs u_\delta(t)\big)\\
	\hspace{16mm}+\displaystyle\delta\int_\Omega\bs u_\delta(t)\cdot\big(\bs v-\bs u_\delta(t)\big)\ge\int_\Omega\bs f(t)\cdot\big(v-\bs u_\delta(t)\big),\quad\forall\bs v\in\K,\text{ for a.e. }t\in(0,T).
	\end{cases}
	\end{equation}
We will also study the system:
	 to find $(\lambda_\delta,\bs u_\delta)$  such that
	\begin{equation}\label{lm}
		\begin{cases}
			\partial_t\bs u_\delta-\Delta\bs u_\delta+\lambda_\delta\bs u_\delta=\bs f\quad\text{ in }Q_T,\\
			\bs u_\delta= 0\quad\text{ on }\partial\Omega\times(0,T),\quad\bs u_\delta(0)=\bs u_0\quad\text{ in }\Omega,\\
			\lambda_\delta\ge \delta,\quad (\lambda_\delta-\delta)(|\bs u_\delta|-1)=0,\quad\text{ in }Q_T.
		\end{cases}
	\end{equation}
	
	The main results of this paper are
	\begin{theorem}\label{T1}
	Given $\bs f\in\bs L^\infty(Q_T)$ and $\bs u_0\in\K$, the variational inequality \eqref{vi} has a unique solution
$$\bs u_\delta \in L^2\big(0,T;\bs H^1_0\big(\Omega)\big)\cap H^1\big(0,T;\bs L^2(\Omega)\big) \cap \bs L^p(Q_T), \qquad 1\leq p \leq \infty.$$

\end{theorem}
	\begin{theorem}\label{T2}
Given $\bs f\in\bs L^\infty(Q_T)$ and $\bs u_0\in\K$, problem \eqref{lm} has a unique solution 
	$$(\lambda_\delta,\bs u_\delta)\in L^p(Q_T)\times\left( L^2\big(0,T;\bs H^1_0\big(\Omega)\big)\cap H^1\big(0,T;\bs L^2(\Omega)\big) \cap \bs L^p(Q_T)\right), \qquad 1\leq p \leq \infty.$$
	Besides, $\bs u_\delta$ solves the variational inequality \eqref{vi}.
	\end{theorem}
	
In Section 2, we define a family of approximating problems, for which we prove existence of solution. We obtain several {\em a priori} estimates that will be important in Section 3.

In Section 3 we prove Theorems \ref{T1} and \ref{T2}.

Section 4 is dedicated to studying a continuous dependence result, both for the variational inequality \eqref{vi} and for the Lagrange problem \eqref{lm}.

\section{Approximating problem}\label{sec2}

Recall that $\delta\in[0,\delta_0]$ and fix $0<\eps<1$. We define the real function $k_{\eps\delta}$ as $k_{\eps\delta}(s)=\delta$ if $s\le0$, $k_{\eps\delta}(s)=\delta+e^\frac{s}\eps-1$ if $s>0$ and a family of approximating problems, setting $\bs u_{\eps\delta}=(u_{1\eps\delta},\ldots, u_{N\eps\delta})$ and $\bs u_0=(u_{10},\ldots,u_{N0})$,
\begin{equation}
\label{ap}
\begin{cases}
	\partial_t\bs u_{\eps\delta}-\Delta \bs u_{\eps\delta}+k_{\eps\delta}(|\bs u_{\eps\delta}|^2-1)\bs u_{\eps\delta}=\bs f\qquad\text{ in }Q_T,\\
\bs u_{\eps\delta}= 0\qquad\text{ on }\partial\Omega\times(0,T),\\
\bs u_{\eps\delta}(0)=\bs u_0\qquad\text{ in }\Omega.
\end{cases}
\end{equation}

Observe that the operator $\Phi_{\eps\delta}(\bs u)=k_{\eps\delta}(|\bs u|^2-1)\bs u$ is monotone. In fact,
\begin{align*}
(\Phi_{\eps\delta}(\bs u)-\Phi_{\eps\delta}(\bs v))\cdot(\bs u-\bs v)&=(k_{\eps\delta}(|\bs u|^2-1)\bs u-k_{\eps\delta}(|\bs v|^2-1)\bs v)\cdot(\bs u-\bs v)\\
&\ge( k_{\eps\delta}(|\bs u|^2-1)|\bs u|-k_{\eps\delta}(|\bs v|^2-1)|\bs v|) (|\bs u|-|\bs v|)\ge0,
\end{align*}
since $ k_{\eps\delta}(s^2-1)s$ is an increasing function for $s\geq 0$.

From now on, whenever there is no confusion, we use $\widehat k_{\eps\delta}$ instead of $k_{\eps\delta}(| \bs u_{\eps\delta}|^2-1)$.

The following result is an immediate consequence of known results for monotone operators (see \cite{Lions1969}).

\begin{proposition}
	If $\bs f\in\bs L^2(Q_T)$ and $\bs u_0\in\K$, problem \eqref{ap} has a unique solution
	$$\bs u_{\eps\delta}\in L^2\big(0,T;\bs H^1_0(\Omega)\big).$$

\end{proposition}

The rest of this section consists on obtaining {\em a priori} estimates independent of $\eps$, which will allow us to pass to the limit and use it to prove the main theorems of this paper.

\begin{proposition} \label{prop2.2} Suppose that $\bs f\in\bs L^\infty(Q_T)$, $\bs u_0\in\K$ and let $\bs u_{\eps\delta}$ be the solution of problem \eqref{ap}. Then there exists a positive constant $C$, depending on $\|\bs f\|_{\bs L^\infty(Q_T)}$ and $\|\bs u_0\|_{\bs L^2(\Omega)}$, but independent of $\eps$ and $\delta$, such that
	\begin{equation}\label{estes}
\|\bs u_{\eps\delta}\|_{L^\infty(0,T;\bs L^2(\Omega))}\le C,\qquad \|\nabla\bs u_{\eps\delta}\|_{\bs L^2(Q_T)}\le C,\qquad \|\widehat k_{\eps\delta}\|_{L^1(Q_T)}\le C.
	\end{equation}
	Besides, for any $1\le p<\infty$, there exists a constant  $C_p>0$, depending on $\|\bs f\|_{\bs L^\infty(Q_T)}$ and $\|\bs u_0\|_{\bs L^2(\Omega)}$, but independent of $\eps$ and $\delta$, such that
	\begin{equation}\label{estup}
\|\bs u_{\eps\delta}\|_{\bs L^p(Q_T)}\le C_p.
	\end{equation}
\end{proposition}
\begin{proof} 
Evaluating the inner product of $\bs u_{\eps\delta}$ with the first equation of \eqref{ap} and integrating over $Q_t=\Omega\times(0,t)$, we get
	\begin{align*}
\tfrac12\int_{Q_t}\partial_t\big(|\bs u_{\eps\delta}|^2\big)+\int_{Q_t}|\nabla \bs u_{\eps\delta}|^2+\int_{Q_t}\widehat k_{\eps\delta}|\bs u_{\eps\delta}|^2&=\int_{Q_t}\bs f\cdot\bs u_{\eps\delta}\\
\nonumber&\le \tfrac12\int_{Q_t}|\nabla \bs u_{\eps\delta}|^2+\tfrac{C^2}2\int_{Q_t}|\bs f|^2,
	\end{align*}
	where $C$ is the Poincar\'e constant,
	obtaining
	\begin{equation*}
		\|\bs u_{\eps\delta}\|^2_{L^\infty(0,T;\bs L^2(Q_T))}+\|\nabla\bs u_{\eps\delta}\|^2_{\bs L^2(Q_T)}+2\|\widehat k_{\eps\delta}|\bs u_{\eps\delta}|^2\|_{L^1(Q_T)}\le  C^2\|\bs f\|^2_{\bs L^2(Q_T)}+\|\bs u_0\|^2_{\bs L^2(\Omega)}.
		\end{equation*}
	
From the previous inequality, we have
	$$\int_{Q_T}\widehat k_{\eps\delta}\big(|\bs u_{\eps\delta}|^2-1\big)+\int_{Q_T}\widehat k_{\eps\delta}\le  C^2\|\bs f\|^2_{\bs L^2(Q_T)}+\|\bs u_0\|^2_{\bs L^2(\Omega)}$$
	and, as
	\begin{align}\label{est2}
	\int_{Q_T}\widehat k_{\eps\delta}\big(|\bs u_{\eps\delta}|^2-1\big)&=\int_{\{|\bs u_{\eps\delta}|>1\}}\widehat k_{\eps\delta}\big(|\bs u_{\eps\delta}|^2-1\big)+\int_{\{|\bs u_{\eps\delta}|\le 1\}}\widehat k_{\eps\delta}\big(|\bs u_{\eps\delta}|^2-1\big)\\
\nonumber	&\ge \delta\int_{\{|\bs u_{\eps\delta}|\le 1\}}\big(|\bs u_{\eps\delta}|^2-1\big),\\
\nonumber	&\ge -\delta|Q_T|,
	\end{align}
	then
	$$\int_{Q_T}\widehat k_{\eps\delta}\le  \tfrac{C^2}2\|\bs f\|^2_{\bs L^2(Q_T)}+\tfrac12\|\bs u_0\|^2_{\bs L^2(\Omega)}+\tfrac{\delta_0}2|Q_T|.$$
	
	Recall that, for $s\ge0$, $k_{\eps\delta}(s)=\delta+e^\frac{s}\eps-1\ge \frac{s^k}{k!\eps^k}$, for any $k\in\N$, and so
	$$\int_{\{|\bs u_{\eps\delta}|>1\}}\tfrac{(|\bs u_{\eps\delta}|^2-1)^k}{k!\eps^k}\le \int_{\{|\bs u_{\eps\delta}|>1\}}\widehat k_{\eps\delta}\le\|\widehat k_{\eps\delta}\|_{L^1(Q_T)}\le C,$$
	concluding that
\begin{align*}
	\int_{Q_T}\big||\bs u_{\eps\delta}|^2-1\big|^k&= \int_{\{|\bs u_{\eps\delta}|\le1\}}\big||\bs u_{\eps\delta}|^2-1\big|^k+\int_{\{|\bs u_{\eps\delta}|>1\}}(|\bs u_{\eps\delta}|^2-1)^k\\
	&\le \int_{\{|\bs u_{\eps\delta}|\le1\}}1+Ck!\eps^k\le |Q_T|+Ck!\eps^k\le |Q_T|+Ck!.
\end{align*}
But then
\begin{equation*}
	\int_{Q_T}|\bs u_{\eps\delta}|^{2k}=\int_{Q_T}|(\bs u_{\eps\delta}|^2-1)+1|^k=2^{k-1}\Big(\int_{Q_T}\big||\bs u_{\eps\delta}\big|^2-1|^k+\int_{Q_T}1\Big)\le 2^{k-1}\big(2|Q_T|+Ck!).
\end{equation*}
Using the continuous inclusion $\bs L^{2k}(Q_T)\subseteq \bs L^p(Q_T)$ for any $k$ such that $p \leq 2k$, we conclude the proof.
\end{proof}

\begin{proposition}\label{prop23}
Suppose that $\bs f\in\bs L^\infty(Q_T)$,  $\bs u_0\in\K$ and let $\bs u_{\eps\delta}$ be the solution of problem \eqref{ap}. Then, for $1\leq p<\infty$, there exists a positive constant $\overline C_p$, depending on $p$, $\|\bs f\|_{\bs L^\infty(Q_T)}$ and $\|\bs u_0\|_{\bs L^2(\Omega)}$, but independent of $\eps$ and $\delta$, such that
	
	\begin{equation*}
	\|\widehat k_{\eps\delta}\|_{L^p(Q_T)}\le \overline C_p,
\end{equation*}
where, for $C_p$ given by \eqref{estup},
\begin{equation}\label{kepsp}
\overline C_p = \Big(\|\bs f\|_{\bs L^\infty(Q_T)}^p(C_p)^p+\tfrac{\delta_0^{p-1}p}2\|\bs u_0\|^2_{\bs L^2(\Omega)}+\delta_0^p p|Q_T|\Big)^\frac1p.
\end{equation}
\end{proposition}
\begin{proof}
Evaluating the inner product of $\widehat k_{\eps\delta}^{p-1}\bs u_{\eps\delta}$ with the first equation of \eqref{ap} and integrating over $Q_t$, we get
\begin{equation*}
	\tfrac12\int_{Q_t}(\widehat k_{\eps\delta})^{p-1}\partial_t\big(|\bs u_{\eps\delta}|^2\big)+\int_{Q_t}-\Delta\bs u_{\eps\delta}\cdot\big((\widehat k_{\eps\delta})^{p-1} \bs u_{\eps\delta}\big)+\int_{Q_t}\big(\widehat k_{\eps\delta}\big)^p|\bs u_{\eps\delta}|^2=\int_{Q_t}(\widehat k_{\eps\delta})^{p-1}\bs f\cdot\bs u_{\eps\delta}.
\end{equation*}
Setting 
\begin{equation}\label{psieps}
	\Psi_{\eps\delta}(s)=\displaystyle\int_0^s(k_{\eps\delta}(\tau))^{p-1}d\tau,
\end{equation} we have 
$\Psi_{\eps\delta}(|\bs u|^2-1)\ge0$ when $|\bs u|>1$ and $\Psi_{\eps\delta}(|\bs u|^2-1)=\delta^{p-1}(|\bs u|^2-1)\ge-\delta^{p-1}$ otherwise. So, as $|\bs u_0|\le1$, we have
\begin{align}\label{kdt}
\nonumber \tfrac12\int_{Q_t}(\widehat k_{\eps\delta})^{p-1}\partial_t\big(|\bs u_{\eps\delta}|^2\big)
&=\tfrac12\int_{Q_t}\partial_t\big(\Psi_{\eps\delta}(|\bs u_{\eps\delta}|^2-1)\big)
=\tfrac12\int_{\Omega}\Psi_{\eps\delta}(|\bs u_{\eps\delta}(t)|^2-1)-\tfrac12\int_{\Omega} \Psi_{\eps\delta}(|\bs u_0|^2-1)\\
&\ge -\tfrac12\delta^{p-1}|\Omega|-\tfrac12\delta^{p-1}\int_\Omega(|\bs u_0|^2-1)=-\tfrac{\delta^{p-1}}2\|\bs u_0\|^2_{\bs L^2(\Omega)}.
\end{align}
As $\bs u_{\eps\delta}=(u_{1\eps\delta},\ldots,u_{N\eps\delta})$, we get
$$\int_{Q_T}-\Delta\bs u_{\eps\delta}\cdot\big((\widehat k_{\eps\delta})^{p-1} \bs u_{\eps\delta}\big)=\sum_{i=1}^N\int_{Q_T}-\Delta u_{i\eps\delta}(\widehat k_{\eps\delta})^{p-1} u_{i\eps\delta}$$
and we are going to evaluate each parcel of the right hand side. To simplify the notations, as $\eps$ and $\delta$ are fixed, we replace $\bs u_{\eps\delta}$ by $\bs u$ and $u_{i\eps\delta}$ by $u_i$. We obtain
	\begin{equation*}
		\int_{Q_T}-\Delta u_i(\widehat k_{\eps\delta})^{p-1} u_i=\int_{Q_T}\widehat k_{\eps\delta}^{p-1}|\nabla u_i|^2+\int_{Q_T}u_i\nabla u_i\cdot(p-1) (\widehat k_{\eps\delta})^{p-2}k'_{\eps\delta}(|\bs u|^2-1)\nabla (|\bs u|^2-1) .
	\end{equation*}
Obviously, the first term of the right hand side is non-negative. We wish to prove that the  other term  is also non-negative. As $(p-1) (\widehat k_{\eps\delta})^{p-2}k'_{\eps\delta}\ge0$, it is enough to show that $\sum_{i=1}^N u\nabla u_\cdot \nabla (|\bs u|^2-1)\ge0$. But
	\begin{equation*}
	\sum_{i=1}^N	u_i\nabla u_i\cdot\nabla (|\bs u|^2-1)=\tfrac12\nabla(|\bs u|^2)\cdot\nabla (|\bs u|^2-1)=\tfrac12\big|\nabla (|\bs u|^2-1)\big|^2\ge0.
	\end{equation*}
	
Then, using \eqref{kdt} and Cauchy-Schwarz and Young inequalities, we have
	\begin{align*}
\int_{Q_T}\big(\widehat k_{\eps\delta}\big)^p|\bs u_{\eps\delta}|^2&\le \int_{Q_T}(\widehat k_{\eps\delta})^{p-1}\bs f\cdot\bs u_{\eps\delta}+\tfrac{\delta^{p-1}}2\|\bs u_0\|^2_{\bs L^2(\Omega)}\\
&\le \int_{Q_T}(\widehat k_{\eps\delta})^{p-1}|\bs u_{\eps\delta}|\|\bs f\|_{\bs L^\infty(Q_T)}+\tfrac{\delta^{p-1}}2\|\bs u_0\|^2_{\bs L^2(\Omega)}\\
&\le \tfrac1{p'}\int_{Q_T}\big(\widehat k_{\eps\delta})^p+\tfrac1p\|\bs f\|_{\bs L^\infty(Q_T)}^p\int_{Q_T}|\bs u_{\eps\delta}|^p+\tfrac{\delta^{p-1}}2\|\bs u_0\|^2_{\bs L^2(\Omega)}
	\end{align*}
	and, arguing as in the inequality \eqref{est2} (where we considered $p=1$), we get
$$\int_{Q_T}\big(\widehat k_{\eps\delta}\big)^p|\bs u_{\eps\delta}|^2=\int_{Q_T}\big(\widehat k_{\eps\delta}\big)^p\big(|\bs u_{\eps\delta}|^2-1\big)+\int_{Q_T}\big(\widehat k_{\eps\delta}\big)^p\ge-\delta^p|Q_T|+\int_{Q_T}\big(\widehat k_{\eps\delta}\big)^p.$$
Then
	$$
\tfrac1p\int_{Q_T}\big(\widehat k_{\eps\delta}\big)^p\le \tfrac1p\|\bs f\|_{\bs L^\infty(Q_T)}^p\int_{Q_T}|\bs u_{\eps\delta}|^p+\tfrac{\delta^{p-1}}2\|\bs u_0\|^2_{\bs L^2(\Omega)}+\delta^p|Q_T|$$
and, using inequality \eqref{estup}, we obtain the conclusion.
\end{proof}

\begin{proposition}
Suppose $\bs f\in\bs L^\infty(Q_T)$,  $\bs u_0\in\K$ and let $\bs u_{\eps\delta}$ be the solution of problem \eqref{ap}. Then there exists a positive constant $C$, depending on $\|\bs f\|_{\bs L^\infty(Q_T)}$ and $\|\nabla\bs u_0\|_{\bs L^2(\Omega)}$, but independent of $\eps$ and $\delta$, such that
	\begin{equation}\label{dt}
	\|\partial_t\bs u_{\eps\delta}\|_{\bs L^2(Q_T)}\le C.
	\end{equation}
\end{proposition}
\begin{proof}
	Evaluating the inner product of $\partial_t\bs u_{\eps\delta}$ with the first the equation of \eqref{ap} and integrating over $Q_t$, we obtain
	$$\int_{Q_t}|\partial_t\bs u_{\eps\delta}|^2+\tfrac12\int_{Q_t}\partial_t(|\nabla\bs u_{\eps\delta}|^2)+\tfrac12\int_{Q_t}\widehat k_{\eps\delta}\partial_t(|\bs u_{\eps\delta}|^2)=\int_{Q_t}\bs f\cdot\partial_t\bs u_{\eps\delta}.$$
	So, with the notation of \eqref{psieps} for $p=2$, using Cauchy-Schwarz and Young inequalities, we obtain
	\begin{multline*}
		\int_{Q_t}|\partial_t\bs u_{\eps\delta}|^2+\tfrac12\int_\Omega|\nabla \bs u_{\eps\delta}(t)|^2+\tfrac12\int_\Omega \Psi_{\eps\delta}(|\bs u_{\eps\delta}(t)|^2-1)\\
		\le\tfrac12\int_{Q_t}|\bs f|^2+\tfrac12\int_{Q_t}|\partial_t\bs u_{\eps\delta}|^2
		+\tfrac12\int_\Omega|\nabla\bs u_0|^2+\tfrac12\int_\Omega\Psi_{\eps\delta}(|\bs u_0|^2-1).
	\end{multline*}
	
	We saw, in the proof of the previous proposition, that $\Psi_{\eps\delta}(|\bs u_{\eps\delta}(t)|^2-1)\ge-\delta$ and $\Psi_{\eps\delta}(|\bs u_0|^2-1)=\delta(|\bs u_0|^2-1)$.
	Therefore,
$$\|\partial_t\bs u_{\eps\delta}\|^2_{\bs L^2(Q_T)}+\|\nabla\bs u_{\eps\delta}\|_{L^\infty(0,T;\bs L^2(\Omega))}^2\le \delta_0|\Omega|+\|\bs f\|_{\bs L^2(Q_T)}^2+\|\nabla\bs u_0\|^2_{\bs L^2(\Omega)}+\delta_0\|\bs u_0\|_{\bs L^2(\Omega)}^2.$$
\end{proof}

\section{The  variational inequality and the Lagrange multiplier system}

In this section we prove Theorems \ref{T1} and \ref{T2}.

\begin{proof} {\em of Theorem \ref{T1}}

The {\em a priori} estimates we obtained in Section \ref{sec2} allow us to pass to the limit, when $\eps\rightarrow0$, of a subsequence of $(u_{\eps\delta})_\eps$:
	\begin{align}\label{convergencias}
		&\nabla \bs u_{\eps\delta}\underset{\eps\rightarrow0}\lraup \nabla\bs u_\delta\quad\text{ in }\bs L^2(Q_T)\text{-weak},&\qquad\partial_t\bs u_{\eps\delta}\underset{\eps\rightarrow0}\lraup \partial_t\bs u_\delta\quad\text{ in }\bs L^2(Q_T)\text{-weak},\\
	\nonumber	&\bs u_{\eps\delta}\underset{\eps\rightarrow0}\longrightarrow\bs u_\delta\quad\text{ in }\bs L^2(Q_T).&
	\end{align}
	
As $\bs u_{\eps\delta}(0)=\bs u_0$ in $\Omega$ and $u_{\eps\delta}=0$ on $\partial \Omega \times (0,T)$, the same is true for $\bs u_\delta$.
	
Setting $A_{\eps\delta}=\{(x,t)\in Q_T:|\bs u_{\eps\delta}|^2-1\ge\eps\}$ and recalling that $e^\frac{s}\eps-1\ge\tfrac{s}\eps$ if $s>0$,
\begin{align*}
	\int_{Q_T}\big(|\bs u_\delta|^2-1\big)^+&=\lim_{\eps\rightarrow0}\int_{Q_T}\big(|\bs u_{\eps\delta}|^2-1\big)^+\le\lim_{\eps\rightarrow0}\Big(\int_{Q_T\setminus A_{\eps\delta}}\eps+\int_{A_{\eps\delta}}\eps\widehat k_{\eps\delta}\Big)\\
	&\le\lim_{\eps\rightarrow0}\big(\eps|Q_T|+\eps\|\widehat k_{\eps\delta}\|_{L^1(Q_T)}\big)=0.
\end{align*}
So $(|\bs u_\delta|^2-1)^+\equiv0$ and thus
\begin{equation*}
\bs u_\delta(t)\in\K\ \text{ for a.e. $t\in(0,T)$}.
\end{equation*}

Given $\bs v\in\K$, evaluating the inner product of $\bs v-\bs u_{\eps\delta}(t)$ with the first the equation of \eqref{ap} and integrating over $\Omega\times(s,t)$, with $0\le s<t\le T$, we obtain
\begin{multline*}
	\int_s^t\int_\Omega\partial_t\bs u_{\eps\delta}\cdot\big(\bs v-\bs u_{\eps\delta}\big)+\int_s^t\int_\Omega\nabla \bs u_{\eps\delta}:\nabla\big(\bs v-\bs u_{\eps\delta}\big)\\
	+\int_s^t\int_\Omega \widehat k_{\eps\delta}\bs u_{\eps\delta}\cdot\big(\bs v-\bs u_{\eps\delta}\big)
	=\int_s^t\int_\Omega\bs f\cdot\big(\bs v-\bs u_{\eps\delta}\big).
\end{multline*}
Observe that
 \begin{equation*}
 \nabla \bs u_{\eps\delta}(t):\nabla\big(\bs v-\bs u_{\eps\delta}(t)\big)=\nabla\big(\bs u_{\eps\delta}(t)-\bs v\big):\nabla\big(\bs v-\bs u_{\eps\delta}(t)\big)+\nabla\bs v:\nabla\big(\bs v-\bs u_{\eps\delta}(t) \big)\le \nabla\bs v:\nabla\big(\bs v-\bs u_{\eps\delta}(t)\big)
 \end{equation*}
and that, using the Cauchy-Schwarz inequality and recalling that $|\bs v|\le 1$ and $\widehat k_{\eps\delta}=\delta$ in $\{|\bs u_{\eps\delta}|<1\}$,
\begin{equation*}
\big( \widehat k_{\eps\delta}-\delta\big)\bs u_{\eps\delta}(t)\cdot\big(\bs v-\bs u_{\eps\delta}(t)\big)\le \big( \widehat k_{\eps\delta}-\delta\big)|\bs u_{\eps\delta}(t)|\big(|\bs v|-|\bs u_{\eps\delta}(t)|\big)\le \big( \widehat k_{\eps\delta}-\delta\big)|\bs u_{\eps\delta}(t)|\big(1-|\bs u_{\eps\delta}(t)|\big)\le0.
 \end{equation*}
So we have
  \begin{multline*}
 \int_s^t	\int_\Omega\partial_t\bs u_{\eps\delta}\cdot\big(\bs v-\bs u_{\eps\delta}\big)+\int_s^t\int_\Omega\nabla \bs v:\nabla\big(\bs v-\bs u_{\eps\delta}\big)\\
 	+\delta\int_s^t\int_\Omega\bs u_{\eps\delta}\cdot\big(\bs v-\bs u_{\eps\delta}\big)
 	\ge\int_s^t\int_\Omega\bs f\cdot\big(\bs v-\bs u_{\eps\delta}\big).
 \end{multline*}
Letting $\eps\rightarrow0$ and because $s$ and $t$ are arbitrary, using the convergences in \eqref{convergencias}, we obtain, for all $\bs v\in\K$,
   \begin{equation*}
 	\int_\Omega\partial_t\bs u_\delta(t)\cdot\big(\bs v-\bs u_\delta(t)\big)+\int_\Omega\nabla \bs v:\nabla\big(\bs v-\bs u_\delta(t)\big)	+\delta\int_\Omega\bs u_\delta(t)\cdot\big(\bs v-\bs u_\delta(t)\big)
 	\ge\int_\Omega\bs f(t)\cdot\big(\bs v-\bs u_\delta(t)\big).
 \end{equation*}
Given any $\bs v\in\K$, we choose now $\bs w=\bs u_\delta(t)+\theta\big(\bs v-\bs u_\delta(t)\big)$ as test function, with $\theta\in(0,1]$, noting that $\bs w\in\K$. Then
    \begin{multline*}
 	\theta\int_\Omega\partial_t\bs u_\delta(t)\cdot\big(\bs v-\bs u_\delta(t)\big)+\theta\int_\Omega\nabla \big(\bs u_\delta(t)+\theta(\bs v-\bs u_\delta)\big):\nabla\big(\bs v-\bs u_\delta(t)\big)\\
 	+\theta\delta\int_\Omega\bs u_\delta(t)\cdot\big(\bs v-\bs u_\delta(t)\big)
 	\ge\theta\int_\Omega\bs f(t)\cdot\big(\bs v-\bs u_\delta(t)\big)
 \end{multline*}
 and, dividing both members by $\theta$ and letting afterwards $\theta\rightarrow0$, we conclude that $\bs u_\delta$ satisfies \eqref{vi}.
 
 Note that,  as $|\bs u_\delta|\le 1$ and using \eqref{estup},
\begin{align*}
\int_{Q_T}|\bs u_{\eps\delta}-\bs u_\delta|^p=\int_{Q_T}|\bs u_{\eps\delta}-\bs u_\delta|^{p-1}|\bs u_{\eps\delta}-\bs u_\delta|&\le \|\bs u_{\eps\delta}-\bs u_\delta\|_{\bs L^{2p-2}(Q_T)}^{p-1}\|\bs u_{\eps\delta}-\bs u_\delta\|_{\bs L^2(Q_T)}\\
&\le \Big(\int_{Q_T}\big(|\bs u_{\eps\delta}|+1\big)^{2p-2}\Big)^\frac12 \|\bs u_{\eps\delta}-\bs u_\delta\|_{\bs L^2(Q_T)}
\end{align*}
and we conclude that
\begin{equation}\label{convp}
\bs u_{\eps\delta}\underset{\eps\rightarrow0}\longrightarrow\bs u_\delta\quad\text{ in }\bs L^p(Q_T).
\end{equation}

It remains to show that the solution of \eqref{vi} is unique, which is straightforward, since the convex $\K$ is fixed. In fact, if $\bs u_{1\delta}$ and $\bs u_{2\delta}$ are two solutions, we may use the second function as test function for the VI solved by the first one, and reciprocally, obtaining immediately that
 $$\tfrac12\int_{\Omega}|\bs u_{1\delta}(t)-\bs u_{1\delta}(t)|^2+\int_{Q_t}|\nabla(\bs u_{1\delta}-\bs u_{2\delta})|^2+\delta\int_{Q_t}|\bs u_{1\delta}-\bs u_{2\delta}|^2\le0,$$
 and the conclusion follows.
\end{proof}

\begin{proof} {\em of Theorem \ref{T2}}

We start by proving the uniqueness of the Lagrange multiplier $\lambda_\delta$. Suppose there exist $\lambda_{1\delta},\lambda_{2\delta} \in L^p(Q_T)$, for some $p \in [1,\infty)$, such that $(\lambda_{1\delta},\bs u_\delta)$ and $(\lambda_{2\delta},\bs u_\delta)$ solve \eqref{lm}. Subtracting the first equality in \eqref{lm} for $\lambda_{2\delta}$ from the one for $\lambda_{1\delta}$ and setting $\xi=\lambda_{1\delta}-\lambda_{2\delta}$, we get
$$\xi\bs u_{\delta}=0\quad\text{a.e. in }Q_T.$$
We know that, in the set $\{|\bs u_{\delta}|<1\}$, we have $\lambda_{1\delta}=\lambda_{2\delta}=\delta$. Otherwise, in the set $\{|\bs u_{\delta}|=1\}$, $\bs u_{\delta}$ is never the null vector and thus, by the equality above, $\xi\equiv 0$ in this set.

Recalling the estimate \eqref{estes} and the uniqueness of $u_\delta$, we have the convergences in \eqref{convergencias} for the whole sequence. From the proof of the previous theorem, we also have the convergence \eqref{convp}, $\bs u_\delta(0)=\bs u_0$ and $\bs u_\delta=0$ on $\partial\Omega\times(0,T)$. Using the estimate \eqref{kepsp}, there exists a unique $\lambda_\delta$ such that $\widehat k_{\eps\delta}\underset{\eps\rightarrow0}\lraup \lambda_\delta$ in $L^p(Q_T)$-weak, for any $p \in [1,\infty)$.
	
Evaluating the inner product of $\bs\varphi\in L^2\big(0,T;\bs H^1_0(\Omega)\big)$ with the first equation of \eqref{ap} and integrating over $Q_T$, we obtain
	$$\int_{Q_T}\partial_t\bs u_{\eps\delta}\cdot\bs \varphi+\int_{Q_T}\nabla\bs u_{\eps\delta}:\nabla\bs\varphi+\int_{Q_T}\widehat k_{\eps\delta}\bs u_{\eps\delta}\cdot\bs\varphi=\int_{Q_T}\bs f\cdot\bs\varphi.$$
	We can pass to the limit for a subsequence, as in \eqref{convergencias}, when $\eps\rightarrow0$, recalling that $\widehat k_{\eps\delta}$ converges weakly to $\lambda_\delta$ in $L^p(Q_T)$, and $\bs u_{\eps\delta}$ converges strongly to $\bs u_\delta$ in $\bs L^p(Q_T)$, getting
	\begin{equation}\label{eqlm}
		\int_{Q_T}\partial_t\bs u_{\delta}\cdot\bs \varphi+\int_{Q_T}\nabla\bs u_{\delta}:\nabla\bs\varphi+\int_{Q_T}\lambda_\delta\bs u_{\delta}\cdot\bs\varphi=\int_{Q_T}\bs f\cdot\bs\varphi.
		\end{equation}
	
As $\widehat k_{\eps\delta}$ converges weakly to $\lambda_\delta$ in $L^p(Q_T)$  and $k_{\eps\delta}\ge \delta$, then $\lambda_\delta\ge \delta$.

 Observing that $(\widehat k_{\eps\delta}-\delta)(|\bs u_{\eps\delta}|-1)\ge0$ and recalling that $\bs u_{\eps\delta}$ converges strongly to $\bs u_\delta$ in $\bs L^p(Q_T)$, then $(\lambda_\delta-\delta)(|\bs u_\delta|-1)\ge0$. But, on the other hand, as $|\bs u_\delta|\le 1$ then $(\lambda_\delta-\delta)(|\bs u_\delta|-1)\leq 0$.

Now we prove that $\bs u_\delta$ solves the variational inequality \eqref{vi}. To do so, note that, given $\bs v\in\K$, 
\begin{equation}\label{aqui}
	(\lambda_\delta-\delta)\bs u_\delta\cdot(\bs v-\bs u_\delta)\le (\lambda_\delta-\delta)|\bs u_\delta|(|\bs v|-|\bs u_\delta|)\le (\lambda_\delta-\delta)|\bs u_\delta|(1-|\bs u_\delta|)=0.
	\end{equation}
So, using $\bs v-\bs u_\delta$ as test function in \eqref{eqlm}, we obtain
	\begin{equation*}
	\int_{Q_T}\partial_t\bs u_{\delta}\cdot(\bs v-\bs u_\delta)+\int_{Q_T}\nabla\bs u_{\delta}:\nabla(\bs v-\bs u_\delta)+\int_{Q_T}((\lambda_\delta-\delta)+\delta)\bs u_{\delta}\cdot(\bs v-\bs u_\delta)=\int_{Q_T}\bs f\cdot(\bs v-\bs u_\delta)
\end{equation*}
and, applying \eqref{aqui}, we get
	\begin{equation*}
	\int_{Q_T}\partial_t\bs u_{\delta}\cdot(\bs v-\bs u_\delta)+\int_{Q_T}\nabla\bs u_{\delta}:\nabla(\bs v-\bs u_\delta)+\delta\int_{Q_T}\bs u_\delta\cdot(\bs v-\bs u_\delta)\ge\int_{Q_T}\bs f\cdot(\bs v-\bs u_\delta).
\end{equation*}
Recalling that the solution of the variational inequality \eqref{vi} is unique, we conclude that
	\begin{equation*}
	\int_{\Omega}\partial_t\bs u_{\delta}(t)\cdot(\bs v-\bs u_\delta(t))+\int_{\Omega}\nabla\bs u_{\delta}(t):\nabla(\bs v-\bs u_\delta(t))+\delta\int_{\Omega}\bs u_\delta(t)\cdot(\bs v-\bs u_\delta(t))\ge\int_{\Omega}\bs f(t)\cdot(\bs v-\bs u_\delta(t)).
\end{equation*}
	
\end{proof}
	
	\begin{remark}
		An interesting question is what happens if we replace the constraint $1$ by a non-negative function $g=g(x,t)$. Several estimates remain true. However, it is not obvious how to extend the estimate in Proposition \ref{prop23} to this case.
		
		As the estimate $\|\widehat k_{\eps\delta}\|_{L^1(Q_T)}	\le C$, $C$ independent of $\eps$ and $\delta$, remains true if we replace $1$ by $g$, the argument of taking the  generalized limit of a subsequence $\widehat k_{\eps\delta}$ in $L^\infty(Q_T)'$, introduced in \cite{AzevedoMirandaSantos2013}, seems possible to apply in the context of this paper, proving existence of a solution  of problem \eqref{lm} in a much weaker sense. 
		\end{remark}
	
	\section{Continuous dependence}
	
	In this section, we study the behaviour of the solution $(\lambda_{n\delta},\bs u_{n\delta})$ of problem \eqref{lm}, with data $(\bs f_n,\bs u_{n0})$ converging to $(\bs f,\bs u_0)$.
		
		\begin{theorem}
		Let, for $n\in\N$, $\bs f_n,\bs f\in\bs L^\infty(Q_T)$, $\bs u_{n0},\bs u_0\in\K$, and suppose that
			\begin{equation}\label{fnf}
				\bs f_n\underset{n}\longrightarrow\bs f\quad\text{ in }\bs L^\infty(Q_T),\qquad\bs u_{n0}\underset{n}\longrightarrow\bs u_0\quad\text{ in }\bs H^1_0(\Omega).
			\end{equation}
			If $(\lambda_{n\delta},\bs u_{n\delta})$ and $(\lambda_\delta,\bs u_\delta)$ denote the unique solutions of problem \eqref{lm} with data $(\bs f_n,\bs u_{n0})$ and $(\bs f,\bs u_0),$ respectively, then, for any $1\leq p<\infty$,
			\begin{equation*}
				\lambda_{n\delta}\underset n\lraup\lambda_\delta \quad\text{ in }L^p(Q_T)\text{-weakly}\quad\text{ and }\quad\bs u_{n\delta}\underset n\longrightarrow\bs u_\delta\quad\text{ in }L^\infty\big(0,T;\bs L^2(\Omega)\big)\cap L^2\big(0,T;\bs H^1_0(\Omega)\big)\cap \bs L^p(Q_T).
			\end{equation*}
		\end{theorem}
		\begin{proof}
			We first prove the convergence of $(\bs u_{n\delta})_n$ to $\bs u_\delta$. To do so, we use $\bs u_{n\delta}(t)$ as test function in the variational inequality \eqref{vi} solved by $\bs u_\delta$ and use $\bs u_\delta(t)$ in the variational inequality solved by $\bs u_{n\delta}$. Subtracting one inequality from the other and integrating over $(0,t)$, we get, setting $\bs z_{n\delta}=\bs u_{n\delta}-\bs u_\delta$,
			\begin{align*}
				\int_{Q_t}\partial_t\bs z_{n\delta}\cdot\bs z_{n\delta}+\int_{Q_t}|\nabla \bs z_{n\delta}|^2+\delta\int_{Q_t}|\bs z_{n\delta}|^2&\le\int_{Q_t}(\bs f_{n}-\bs f)\cdot\bs z_{n\delta}\\
				&\le \tfrac{C^2}2\int_{Q_t}|\bs f_n-\bs f|^2+\tfrac12\int_{Q_t}|\nabla \bs z_{n\delta}|^2.
			\end{align*}
			But then
			$$\tfrac12\int_\Omega|\bs z_{n\delta}(t)|^2+\tfrac12\int_{Q_t}|\nabla \bs z_{n\delta}|^2+\delta\int_{Q_t}|\bs z_{n\delta}|^2\le \tfrac{C^2}2\int_{Q_t}|\bs f_n-\bs f|^2+\tfrac12\int_\Omega|\bs u_{n0}-\bs u_0|^2,$$
			concluding the convergence of $\bs u_{n\delta}$ to $\bs u_\delta$ in $L^\infty\big(0,T;\bs L^2(\Omega)\big)\cap L^2\big(0,T;\bs H^1_0(\Omega)\big)$. Observe that, as $|\bs u_{n\delta}|\le 1$ and $|\bs u_\delta|\le 1$, we have, for $p\geq 2$,
			\begin{equation}\label{undp}
				\int_{Q_T}|\bs u_{n\delta}-\bs u_\delta|^p\le 2^{p-2}\int_{Q_T}|\bs u_{n\delta}-\bs u_\delta|^2\underset n\longrightarrow0.
				\end{equation}
			
			Considering now the sequence of Lagrange multipliers $(\lambda_{n\delta})_n$, we know that $(\lambda_{n\delta},\bs u_{n\delta})$ solves problem \eqref{lm}. In particular, we have
			\begin{equation}\label{vilim}
				\int_{Q_T}\partial_t\bs u_{n\delta}\cdot\bs\varphi+\int_{Q_T}\nabla \bs u_{n\delta}\cdot\nabla\bs\varphi+\int_{Q_T}\lambda_{n\delta}\bs u_{n\delta}\cdot\bs\varphi=\int_{Q_T}\bs f_n\cdot\bs\varphi,\qquad\bs u_{n\delta}(0)=\bs u_{n0}.
				\end{equation}
			
			Recall that $\widehat k_{n\eps\delta} = k_{\eps\delta}(| \bs u_{n\eps\delta}|^2-1)$, where $u_{n\eps\delta}$ is the unique solution of the approximating problem \eqref{ap} with data $(\bs f_n, \bs u_{n0})$, and $\lambda_{n\delta}$ is the weak limit in $L^p(Q_T)$ of $(\widehat k_{n\eps\delta})_\eps$, when $\eps\rightarrow0$. Thus, using \eqref{kepsp},
			$$\|\lambda_{n\delta}\|_{L^p(Q_T)}\le \varliminf_{\eps\rightarrow0}\|\widehat k_{n\eps\delta}\|_{L^p(Q_T)}\le \Big(\|\bs f_n\|_{\bs L^\infty(Q_T)}^p(C_{np})^p+\tfrac{\delta_0^{p-1}p}2\|\bs u_{n0}\|^2_{\bs L^2(\Omega)}+\delta_0^p p|Q_T|\Big)^\frac1p,$$
			where the constant $C_{np}$  depends only on $p, \|\bs f_n\|_{\bs L^\infty(Q_T)}$ and  $\|\bs u_{n0}\|_{\bs L^2(\Omega)}.$ So, by the assumption \eqref{fnf}, there exists $C>0$ independent of $n$ such that
			$$\|\lambda_{n\delta}\|_{L^p(Q_T)}\le C.$$
			
			By \eqref{estes} and \eqref{dt}, we have
			$$\|\nabla\bs u_{n\eps\delta}\|_{\bs L^2(Q_T)}\le C_1,\qquad\|\partial_t\bs u_{n\eps\delta}\|_{\bs L^2(Q_T)}\le C_2,$$
			where $C_1=C_1(\|\bs f_n\|_{\bs L^\infty(Q_T)},\|\bs u_{n0}\|_{\bs L^2(\Omega)})$ and $C_2=C_2(\|\bs f_n\|_{\bs L^\infty(Q_T)},\|\nabla\bs u_{n0}\|_{\bs L^p(\Omega)})$. So the assumption \eqref{fnf} implies the uniform boundedness (independent of $n$) of $\|\nabla\bs u_{n\eps\delta}\|_{\bs L^2(Q_T)}$ and $\|\partial_t\bs u_{n\eps\delta}\|_{\bs L^2(Q_T)}$. We have that $\nabla \bs u_{n\delta}$ is the weak limit of $\nabla \bs u_{n\eps\delta}$ and $\partial_t \bs u_{n\delta}$ is the weak limit of $\partial_t \bs u_{n\eps\delta}$ in $\bs L^2(Q_T)$. Then $\|\nabla\bs u_{n\delta}\|_{\bs L^2(Q_T)} \leq C_1$ and $\|\partial_t\bs u_{n\delta}\|_{\bs L^2(Q_T)} \leq C_2$.
			
			Using the previous uniform estimates and \eqref{undp}, there exists $\bs u_\delta\in L^2\big(0,T;\bs H^1_0(\Omega)\big) \cap H^1\big(0,T;\bs L^2(\Omega)\big) \cap \bs L^p(Q_T)$ and $\lambda_\delta\in L^p(Q_T)$ such that, for a subsequence,
			\begin{eqnarray*}
&\nabla\bs u_{n\delta}\underset n\lraup\nabla\bs u_\delta\quad\text{ in }\bs L^2(Q_T)\text{-weak},\qquad\qquad &\bs u_{n\delta}\underset n\longrightarrow\bs u_\delta\quad\text{ in }\bs L^p(Q_T),\\
&\partial_t\bs u_{n\delta}\underset n\lraup\partial_t\bs u_\delta\quad\text{ in }\bs L^2(Q_T)\text{-weak},\qquad\qquad &\lambda_{n\delta}\underset n\lraup\lambda_\delta\quad\text{ in }L^p(Q_T)\text{-weak}.
			\end{eqnarray*}
			
			Letting $n\rightarrow\infty$ in \eqref{vilim}, we conclude that
						\begin{equation*}
				\int_{Q_T}\partial_t\bs u_{\delta}\cdot\bs\varphi+\int_{Q_T}\nabla \bs u_{\delta}\cdot\nabla\bs\varphi+\int_{Q_T}\lambda_{\delta}\bs u_{\delta}\cdot\bs\varphi=\int_{Q_T}\bs f_{\delta}\cdot\bs\varphi,\qquad\bs u_{\delta}(0)=\bs u_0.
			\end{equation*}
			Moreover, as $u_{n\delta}=0$ on $\partial\Omega \times (0,T)$, we conclude that $u_{\delta}=0$ on $\partial\Omega \times (0,T)$ and, as $(\lambda_{n\delta}-\delta)(|\bs u_{n\delta}|-1)=0$, the convergences above are enough to conclude that $(\lambda_{\delta}-\delta)(|\bs u_{\delta}|-1)=0$. So the conclusion follows.
\end{proof}

\section*{Acknowledgements}
	
\noindent The authors  were partially financed by Portuguese Funds
through FCT (Funda\c c\~ao para a Ci\^encia e a Tecnologia) within the Projects UIDB/00013/2020 and UIDP/00013/2020.

\end{document}